\documentclass[a4paper,12pt]{article}
\usepackage{a4}
\usepackage{amsmath}
\usepackage{amsfonts}
\usepackage{amsthm}
\usepackage{amssymb}
\usepackage{graphicx}
\usepackage[all,cmtip]{xy}
\makeatletter

\newtheorem{thm}{Theorem}

\newtheorem{prop}[thm]{Proposition}

\newtheorem{lemma}[thm]{Lemma}

\newtheorem{theorem}[thm]{Theorem}

\begin{document}
\begin{center}
\large \bf {Coloring of a Digraph}
\end{center}
$$
\textbf{E. Sampathkumar}
$$
\begin{center}
\small Department of Studies in Mathematics\\
\small University of Mysore\\
\small Manasagangothri\\
\small MYSORE - 570006, INDIA.
\end{center}
\begin{center}
\small{e-mail: esampathkumar@gmail.com}\\
\end{center}

\begin{abstract}

\qquad A \emph{coloring} of a digraph $D=(V,E)$ is a coloring of its vertices following the rule:
Let $uv$ be an arc in $D$. If the tail $u$ is colored first, then the head $v$ should receive a color different from that of $u$. The \emph{dichromatic number} $\chi_d(D)$ of $D$ is the minimum number of colors needed in a coloring of $D$. Besides obtaining many results and bounds for $\chi_d(D)$ analogous to that of chromatic number of a graph, we prove $\chi_d(D)=1$ if $D$ is acyclic. New notions of sequential colorings of graphs/digraphs are introduced. A characterization of acyclic digraph is obtained interms of $L$-matrix of a vertex labeled digraph.
\end{abstract}
\noindent \textit{MATHEMATICS} Subject Classification: 05C15, 05B20.


\section{Introduction}

\qquad A \emph{coloring} of a digraph $D=(V,E)$ is a coloring of its vertices by the following rule:\\
Let $uv$ be an arc in $D$. If the tail $u$ is colored first, then the head $v$ should receive a color different from that of $u$. But, if $v$ is colored first, then $u$ may or may not receive the color of $v$.\\
The \emph{dichromatic number} $\chi_d(D)$ of $D$ is the minimum number of colors needed in a coloring of $D$.\\

\qquad A \emph{coloring} of a graph is a coloring of its vertices such that no two adjacent vertices receive the same color. The \emph{chromatic number} $\chi(G)$ of a graph $G$ is the minimum number of colors needed in a coloring of $G$. \\
The above definitions of colorings of a graph/digraph are equivalent to the following.\\
Let $G=(V,E)$ be a graph or digraph of order $p$, and $V=\{v_1,v_2,\ldots,v_p\}$. A \emph{coloring} of $G$ with colors $c_1,c_2,...,$ is a sequence of ordered pairs $ (v_1,c_1), (v_2,c_2),...,(v_p,c_p)$ such that for $i<j$, $c_j \neq c_i$ if $v_iv_j$ is an edge/arc according as $G$ is a graph/digraph. The \emph{chromatic/dichromatic} number of $G$ is the minimum number of colors needed in a coloring of $G$ according as $G$ is a graph or a digraph.\\
\unitlength 1mm 
\linethickness{0.4pt}
\ifx\plotpoint\undefined\newsavebox{\plotpoint}\fi 
\begin{picture}(96.5,102)(0,0)
\put(51.5,96.25){\circle*{1.5}}
\put(32.75,71){\circle*{1.5}}
\put(70.5,71){\circle*{1.5}}
\put(41.625,83.625){\vector(-3,-4){.07}}\multiput(51,96.25)(-.0337230216,-.0454136691){556}{\line(0,-1){.0454136691}}
\put(51.25,71){\vector(1,0){.07}}\put(32.25,71){\line(1,0){38}}
\put(51.25,91){$v_1$}
\put(36.5,73.25){$v_2$}
\put(64.75,73.5){$v_3$}
\put(50.75,102){$1$}
\put(28.5,65.5){$2$}
\put(72.25,66){$3$}
\put(30,58.75){$s: (v_1,1) , (v_2,2),(v_3,3)$}
\put(19,41.5){\circle*{1.5}}
\put(1.75,18){\circle*{1.5}}
\put(35.75,18){\circle*{1.5}}
\put(10.375,29.875){\vector(-3,-4){.07}}\multiput(19.25,41.25)(-.0336812144,-.0431688805){527}{\line(0,-1){.0431688805}}
\put(18.375,18.5){\vector(1,0){.07}}\put(1.5,18.5){\line(1,0){33.75}}
\put(27.125,29.875){\vector(-3,4){.07}}\multiput(35.25,18.5)(-.0337136929,.0471991701){482}{\line(0,1){.0471991701}}
\put(18.5,35.25){$v_1$}
\put(7,21){$v_2$}
\put(30.75,20.75){$v_3$}
\put(19,46.5){$1$}
\put(0,14.5){$2$}
\put(37.5,15.75){$2$}
\put(-3,10){$s: (v_1,1),(v_3,2),(v_2,2)$}
\put(50.5,53.75){$(a)$}
\put(14,4){$(b)$}
\put(75.75,43.25){\circle*{1.5}}
\put(60,19.75){\circle*{1.5}}
\put(67.625,31.625){\vector(-2,-3){.07}}\multiput(75.25,43.25)(-.0337389381,-.0514380531){452}{\line(0,-1){.0514380531}}
\put(77.25,20.125){\vector(1,0){.07}}\multiput(60,20)(4.3125,.03125){8}{\line(1,0){4.3125}}
\put(75.5,38.75){$v_1$}
\put(65.25,22.75){$v_2$}
\put(88.75,23){$v_3$}
\put(75.5,47.75){$1$}
\put(56,16.5){$1$}
\put(96.5,17){$1$}
\put(55,10){$s:(v_3,1),(v_2,1),(v_1,1)$}
\put(76.25,4){$(c)$}
\put(45,0){$Figure 1$}
\put(61.25,83.5){\vector(3,-4){.07}}\multiput(52.25,95.75)(.0337078652,-.0458801498){534}{\line(0,-1){.0458801498}}
\put(85.125,31.875){\vector(3,-4){.07}}\multiput(76,43.25)(.0337338262,-.042051756){541}{\line(0,-1){.042051756}}
\put(93.75,20.75){\circle*{1.5}}
\end{picture}
\\

\qquad Figure 1 exhibits the colorings of the digraphs following the above rule. The sequences of ordered pairs indicate the order in which the vertices are colored using the above rule. We note that the dichromatic number of the $3$-cycle is 2, and that of semi-cycle of length 3 is 1. The above examples show that the number of colors needed to a color a digraph $D$ depends on the sequence of vertices we choose to color them.

\qquad It appears that there is no standard notion of digraph coloring as compared to the usual coloring of graphs. However, there are a few different notions of digraph coloring ( see for example \cite{G}, \cite{S}, \cite{Ah}, \cite{N}, \cite{ES}).\\

 For any definitions on graphs and digraphs, we refer the book \cite{F}. In this paper, we initiate a study of dichromatic number $\chi_d(D)$ of a digraph $D$, and obtain many results and bounds for $\chi_d(D)$ analogous to those of chromatic number of a graph.\\

\qquad Also, we introduce new coloring parameters of graphs and digraphs, called sequential coloring numbers, closely related to the chromatic and dichromatic numbers.\\\\
\section{Some Basic Results}
\begin{prop}
a) For a directed path $P_n$ on n vertices, $\chi_d(P_n)=1$. \\
b) For any directed cycle $C_n$ on n vertices, $n\geq3$, $\chi_d(C_n)=2$.\\
\end{prop}
\begin{proof} a) Let $v_1v_2 \ldots v_n$  be a directed path. Color the vertices in the order $v_n,v_{n-1},\ldots,v_1$ with color 1. \\
b)Let $C_n:v_1v_2 \ldots v_n$ be a directed cycle. First color $v_1$ and $v_n$ with color 1, and then color the vertices $v_2,v_3,\ldots,v_{n-1}$ alternately with colors 2 and 1.\end{proof}
\begin{prop} If $C_n:v_1v_2 \ldots v_n$ is a semi-cycle which is not a directed cycle, then $\chi_d(C_n)=1$.\end{prop}
\begin{proof} Clearly, $C_n$ has a sink say $v_1$, where $v_2v_1$ and $v_nv_1$ are arcs.  Let $v_1v_2 \ldots v_i$, $2\leq i \leq n-1$ be a maximal directed path. We color the vertices $v_{i-1},v_{i-2},\ldots,v_2$ by the color 1, in this order. Now, we look for a maximal directed path from $v_i$ to $v_j$, where $i<j$, and color the vertices $v_j,v_{j-1},\ldots,v_i$ by the color 1, in this order. Note that $v_j$ may be $v_i$. If some more vertices are left out, we look for a vertex $v_k$, $j<k$, such that $v_kv_{k-1}...v_j$ is a maximal directed path. We color the vertices of this path starting from $v_j$ to $v_k$. We continue this process and color all vertices of $C_n$ with color 1.\end{proof}
\textbf{Remark 1}: Given a sequence of vertices of a digraph $D$ it is easy to find the minimum number of colors needed to color the vertices in the order given by the sequence. But, given a coloring of a digraph, it may be difficult in general to find a sequence of vertices which gives that coloring.\\

For a digraph $D$, the \emph{underlying graph} of $D$, denoted by $G(D)$ has the same vertex set as $D$, where two vertices $u$ and $v$ are adjacent if $uv$ or $vu$ is an arc in $D$.\\

Clearly, for any digraph $D$, $$\chi_d(D)\leq \chi(G(D)),$$ where $\chi(G(D))$ is the chromatic number of the graph $G(D)$. Further, $\chi_d(D)=\chi(G(D))$ if, and only if, every arc in $D$ is symmetric.
\section{Bounds for dichromatic number}
\qquad In a digraph $D$, the \emph{indegree} $ind(v)$ is the number of arcs having $v$ as their head, and the \emph{outdegree} $od(v)$ is the number of arcs having $v$ as their tail.\\

\qquad For a digraph $D=(V,E)$, let $\Delta_{in}(D)=$ $ max\{ind(v): v \in V\}$, and $\Delta_{od}(D)=$ $ max\{od(v): v \in V\}$. \\

\begin{prop} For any digraph $D$ of order $p$ without symmetric arcs $$\chi_d(D) \leq p- \Delta_{in}(D).$$ \end{prop}
\begin{proof} Let $v$ be a vertex with $ind(v)= \Delta_{in}(D)$. First, color $v$ with color 1, and then color all tails of arcs having $v$ as their head, with color 1. Thus, after coloring $\Delta_{in}(D)+1$ vertices with color 1, the rest of the vertices can be colored with $p- \Delta_{in}(D)+1$ distinct colors. Thus, all vertices in $D$ can be colored with $p- \Delta_{in}(D)$ colors, and $\chi_d(D) \leq p- \Delta_{in}(D)$.\end{proof}

\qquad Note that Proposition 3 is not true if $D$ has symmetric arcs. For example, in Figure 2, $\Delta_{in}(D) = 2$, $\chi_d(D) = 2$, and $p - \Delta_{in}(D) =1$.\\

\unitlength 1mm 
\linethickness{0.4pt}
\ifx\plotpoint\undefined\newsavebox{\plotpoint}\fi 
\begin{picture}(66,62)(0,0)
\put(54.25,55.5){\circle*{1.5}}
\put(37.5,26.5){\circle*{1.5}}
\put(69.75,26.75){\circle*{1.5}}
\put(46,41.125){\vector(1,2){.07}}\multiput(37.75,26.5)(.0336734694,.0596938776){490}{\line(0,1){.0596938776}}
\put(62,41.25){\vector(1,-2){.07}}\multiput(54.25,55.75)(.0336956522,-.0630434783){460}{\line(0,-1){.0630434783}}
\put(53.875,26.375){\vector(1,0){.07}}\multiput(38,26.25)(3.96875,.03125){8}{\line(1,0){3.96875}}
\put(54.25,16.5){\vector(-1,0){.07}}\qbezier(69.5,27)(54.875,6.25)(37.75,26.5)
\put(54.25,62){$v_1$}
\put(29.25,24.25){$v_2$}
\put(76,24){$v_3$}
\put(64.5,29.25){$1$}
\put(43,29.25){$2$}
\put(54.25,49.25){$1$}
\put(43.5,-5){$Figure$ $2$}
\put(30,5){$s:(v_3,1), (v_2,2), (v_1,1)$}
\end{picture}
\\

A set $S$ of vertices in a digraph $D$ is \emph{independent} if there is no arc between any two vertices in $S$.\\\\
One can obtain a stronger result.\\
\begin{prop} Let $D=(V,E)$ be a digraph of order $p$ without symmetric arcs, and $S$ be an independent set of vertices. Then $$\chi_d(D) \leq p- \sum_{v \in S} ind(v) + \mid S \mid .$$ \end{prop}

\qquad A digraph $D$ is a \emph{tree} if the underlying graph $G(D)$ of $D$ is a tree. A vertex $v$ in a tree $D$ is a pendant vertex if it is a pendant vertex in $G(D)$. \\

We now prove that all acyclic digraphs have dichromatic number equal to 1. First we prove,\\
\begin{lemma}If a digraph $D$ is a tree, then $\chi_d(D)=1$.
\end{lemma}
\begin{proof} We prove the result by induction on the number of vertices in $D$. Clearly, the result is true if $D$ has 2 or 3 vertices. Suppose it is true if $D$ has $p$ vertices where $p\geq 3$, and let $D$ be a tree with $p+1$ vertices. There is a pendant vertex $v$ in $D$ such that for some other vertex $u$ in $D$, either $uv$ or $vu$ is an arc.\\
\emph{Case 1}. \emph{$vu$ is an arc}. \\
By inductive hypothesis, all vertices in $D-v$ can be colored with the same color, say 1. Now, giving color 1 to $v$ in $D$, one can obtain a coloring of $D$ in which all vertices have the same color 1.\\
\emph{Case 2}. \emph{$uv$ is an arc}.\\
Consider a 1-coloring of $D-v$. In this coloring, $u$ is given color 1. Remove the color given to $u$. Color $v$ with color 1, and then give color 1 to $u$. This gives a 1- coloring of $D$. This completes the proof.
\end{proof}

Let $D$ be a digraph, and $u$, $v$ be two vertices in $D$, such that $uv$ is not an arc in $D$. The digraph obtained by adding the arc $uv$ to $D$ is denoted by $D+uv$.\\
\begin{lemma} Let $u$, $v$ be two vertices in a digraph $D$ such that $\chi_d(D)=1$, and there is no directed path from $u$ to $v$. Then there exists a 1-coloring of $D$ such that $u$ is colored first and then $v$.\end{lemma}
\begin{proof} If $D$ is a semi-path, one can easily see that this is true. Otherwise, we prove the result by induction on the number $p$ of vertices in $D$ . The result is true if $D$ has 3 or 4 vertices. Let the result be true for all digraphs $D$ with $\chi_d(D)=1$ with $p$ vertices where $p\geq4$.\\

 Suppose $D'$ is a digraph of order $p+1$ with $\chi_d(D')=1$. Let $u$, $v$ be two vertices in $D'$ such that there is no directed path from $u$ to $v$ in $D'$. Suppose $w$ be a vertex in $D'$ which is not on a $u-v$ semi-path. The digraph $D=D'-w$ has $p$ vertices, and $\chi_d(D)=1$. By inductive hypothesis, there exists a 1-coloring of $D$ such that $u$ is given color 1 before $v$ is colored.\\
 \emph{Case 1}. $w$ is not the head of any arc in $D'$. In this case, any 1-coloring of $D$ can be extended to 1- coloring of $D'$ by giving color 1 to $w$, where $u$ is colored first, and then $v$.\\
 \emph{Case 2}. The vertex $w$ is the head of arcs say, $w_1w,w_2w,\ldots,w_kw$. Consider a 1-coloring of $D$ in which $u$ is colored first before $v$. In this coloring, remove the colors of the vertices $w_1,w_2,\ldots,w_k$. Then in $D'$, give color 1 to $w$, and then color the vertices $w_1,w_2,\ldots,w_k$. This gives the required coloring of $D'$.
\end{proof}
\begin{theorem}
 Let $D$ be an acyclic digraph. Then $\chi_d(D)=1$.
\end{theorem}
\begin{proof} The result is true if $D$ is a tree (Lemma 5). Suppose, $D$ is not a tree. Let $D'$  be the spanning tree obtained from $D$ by removing arcs $u_1v_1,u_2v_2,\ldots,u_kv_k$. In the tree $D'$, there is no directed path from $v_k$ to $u_k$. For otherwise, we have a directed cycle in $D$, which is not true. By Lemma 5, there is a 1-coloring of $D'$ such that $v_k$ is colored first, and then $u_k$. The same 1-coloring of $D'$ gives a 1-coloring of $D''=D'+u_kv_k$. Putting back the arcs $u_{k-1}v_{k-1},\ldots,u_1v_1$ to $D''$ one by one, to get back $D$, we find with the same type of argument, $\chi_d(D)=1$.
\end{proof}
\section{$c$-Independent Sets}
\qquad Consider a coloring $C$ of a digraph $D=(V,E)$. A set $S$ of vertices in $D$ is \emph{color independent} or \emph{c-independent} with respect to $C$ if all vertices in $S$ have the same color. The \emph{c-independence number} $\beta_{oc}(D)$ of $D$ is the maximum cardinality of a $c$-independent set.\\

One can also define the dichromatic number $\chi_d(D)$ of a digraph $D$ as the minimum order of a partition of the vertex set $V$ of $D$ into $c$-independent sets. Such a partition is called \emph{dichromatic partition} of $D$.\\

As in graphs, one can easily establish the following result.\\
\begin{prop}For a digraph $D$ of order $p$,
$$\frac{p}{\beta_{oc}(D)} \leq \chi_d(D)\leq p-\beta_{oc}(D)+1.$$\end{prop}
These bounds are tight. For example, let $D$ be acyclic digraph of order $p$. Then $\chi_d(D)=1$ and $\beta_{oc}(D)=p$.
Also, for the $3$-cycle in the Figure 3 below, $\chi_d(c_3)=2$, $\beta_{oc}(D)=2$.\\

\unitlength 1mm 
\linethickness{0.4pt}
\ifx\plotpoint\undefined\newsavebox{\plotpoint}\fi 
\begin{picture}(68.75,45)(0,0)
\put(49.75,39.5){\circle*{1.5}}
\put(33,14.5){\circle*{1.5}}
\put(65,15){\circle*{1.5}}
\put(33,14.5){\circle*{1.5}}
\put(49,15){\vector(1,0){.07}}\put(33,15){\line(1,0){32}}
\put(57.25,27.375){\vector(-2,3){.07}}\multiput(65,15)(-.0336956522,.0538043478){460}{\line(0,1){.0538043478}}
\put(41.25,27.5){\vector(-2,-3){.07}}\multiput(49.5,39.75)(-.0336734694,-.05){490}{\line(0,-1){.05}}
\put(48,33.75){$v_1$}
\put(36,17){$v_2$}
\put(58.25,17){$v_3$}
\put(49,42){$1$}
\put(28,11.75){$2$}
\put(68.75,12.75){$1$}
\put(26,3.5){$s:(v_1,1),(v_3,1),(v_2,2)$}
\put(43,-3){$Figure$ $3$}
\end{picture}
\newpage
\section{$c$-Bipartite digraphs}
\qquad A digraph $D=(V,E)$ is \emph{c-bipartite} if $V$ can be partitioned into two $c$-independent sets.
Trivially, an acyclic digraph is $c$-bipartite. We now characterize $c$-bipartite digraphs.\\

Let $C_n:v_1v_2 \ldots v_nv_1$, be a directed cycle. If each arc $v_iv_{i+1}$, $1 \leq i \leq n-1$, and $v_nv_1$is symmetric, then $C_n$ is called a \emph{symmetric cycle} of length $n$.\\
\begin{prop} The following statements are equivalent for a digraph $D$.\\
(i) $D$ is $c$-bipartite.\\
(ii) $\chi_d(D)=2$.\\
(iii) $D$ does not have an odd symmetric cycles.\end{prop}

The proof is more or less similar to the corresponding result for graphs, and we omit the same.\\

If $C_n$ is a dicycle of length $n\geq 2$, then $C_n$ is $c$-bipartite, since $\chi(C_n)=2$, for all $n \geq 2$.
\section{Sequential Coloring or $s$-Coloring}
\qquad While coloring the vertices of a graph or digraph, we invariably choose a sequence of vertices to color them. This motivates one to define sequential coloring or $s$-coloring, and $s$-coloring number of a graph or digraph as follows:\\

Let $G=(V,E)$ be a graph or digraph of order $p$, and $V$ be the vertex set of $G$. Suppose $s: v_1,v_2,\ldots,v_p$ is a sequence of vertices in $V$. A $s$-\emph{coloring} of $G$ with colors $c_1,c_2,\ldots,$ is a sequence of ordered pairs\\
$(v_1,c_1), (v_2,c_2),...,(v_p,c_p)$  \hskip 2 cm ...$(1)$\\
such that for  $i<j$, $c_j \neq c_i$  if $v_iv_j$ is an edge or an arc according as $G$ is a graph or digraph.\\

The \emph{$s$-coloring number} of $G$ is the minimum number of colors that can appear in the sequence $(1)$. When $G$ is a graph, the $s$-coloring number is called the \emph{$s$-chromatic number} of $G$, and it is denoted by $\chi_s(G)$. When $G$ is a digraph, the $s$-coloring number of $G$ is called the \emph{$s$-dichromatic number} of $G$, and it is denoted by $\chi_{sd}(G)$. When $G$ is graph, the sequence (1) is \emph{chromatic} if $\chi_s(G) = \chi(G)$,  the chromatic number of $G$. Similarly, when $G$ is a digraph, the sequence (1) is \emph{dichromatic} if $\chi_{sd}(G)= \chi_d(G)$, the dichromatic number of $G$.

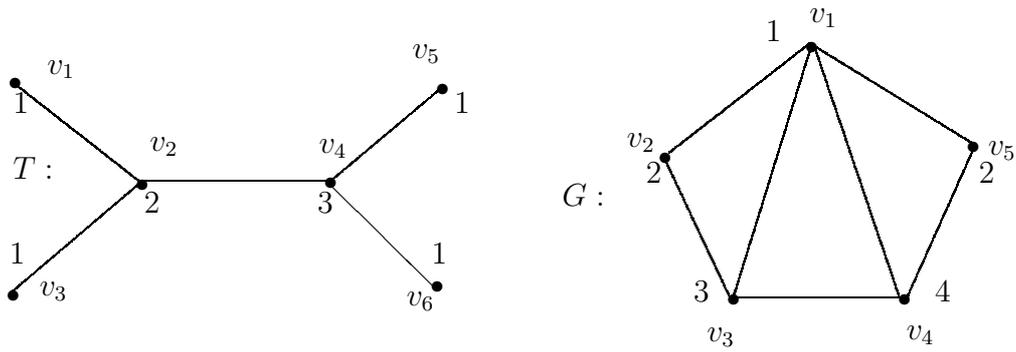
\unitlength 1mm 
\linethickness{0.4pt}
\ifx\plotpoint\undefined\newsavebox{\plotpoint}\fi 
\begin{picture}(133.25,70.75)(0,0)
\put(18.75,46.5){\circle*{1.5}}
\put(43.5,46.75){\circle*{1.5}}
\put(58.25,59.25){\circle*{1.5}}
\put(57.5,33){\circle*{1.5}}
\put(1.75,31.75){\circle*{1.5}}
\put(2,60){\circle*{1.5}}
\multiput(1.5,32.25)(.0393835616,.0336757991){438}{\line(1,0){.0393835616}}
\put(18.75,47){\line(1,0){24.75}}
\multiput(43.5,47)(.0397574124,.0336927224){371}{\line(1,0){.0397574124}}
\multiput(1.75,60)(.042746114,-.0336787565){386}{\line(1,0){.042746114}}
\put(43.5,46.5){\line(1,-1){13.5}}
\put(6.25,61.25){$v_1$}
\put(19.75,51){$v_2$}
\put(54.25,63.25){$v_5$}
\put(5.25,31.75){$v_3$}
\put(42,50.75){$v_4$}
\put(53.5,30.5){$v_6$}
\put(1.75,56){$1$}
\put(1.25,36){$1$}
\put(19,42.75){$2$}
\put(41.75,42.75){$3$}
\put(59.75,56){$1$}
\put(56.75,36){$1$}
\put(1.75,47.25){$T:$}
\put(106.75,64.75){\circle*{1.5}}
\put(87.5,50){\circle*{1.5}}
\put(128,51.5){\circle*{1.5}}
\put(96.5,31.25){\circle*{1.5}}
\put(119,31.25){\circle*{1.5}}
\multiput(88,50.75)(.0430232558,.0337209302){430}{\line(1,0){.0430232558}}
\multiput(106.5,65.25)(.0542394015,-.0336658354){401}{\line(1,0){.0542394015}}
\multiput(128.25,51.75)(-.0337078652,-.0758426966){267}{\line(0,-1){.0758426966}}
\put(119.25,31.5){\line(-1,0){23.25}}
\multiput(96,31.5)(-.0337301587,.0724206349){252}{\line(0,1){.0724206349}}
\multiput(106.5,64.25)(-.0337171053,-.109375){304}{\line(0,-1){.109375}}
\multiput(107.25,64.5)(.0336391437,-.1001529052){327}{\line(0,-1){.1001529052}}
\put(106.5,68){$v_1$}
\put(82.5,51.5){$v_2$}
\put(93,25.75){$v_3$}
\put(119.25,26){$v_4$}
\put(130,50.5){$v_5$}
\put(100.75,65.5){$1$}
\put(85,46.75){$2$}
\put(91.25,31){$3$}
\put(123,31){$4$}
\put(128.75,46.75){$2$}
\put(74,43.75){$G:$}
\put(-2,17){$s_1:(v_1,1),(v_3,1),(v_5,1),(v_6,1),(v_2,2),(v_4,3)$}
\put(29.25,10){$(a)$}
\put(84,17){$s_2: (v_1,1),(v_2,2),(v_5,2),(v_3,3),(v_4,4)$}
\put(105.5,10){$(b)$}
\put(64,4.25){$Figure$ $4$}
\end{picture}

For example, in Figure 4(a), $\chi_{s_1}(T)= 3$, and in Figure 4(b) $\chi_{s_2}(G)=4$.
\\\\

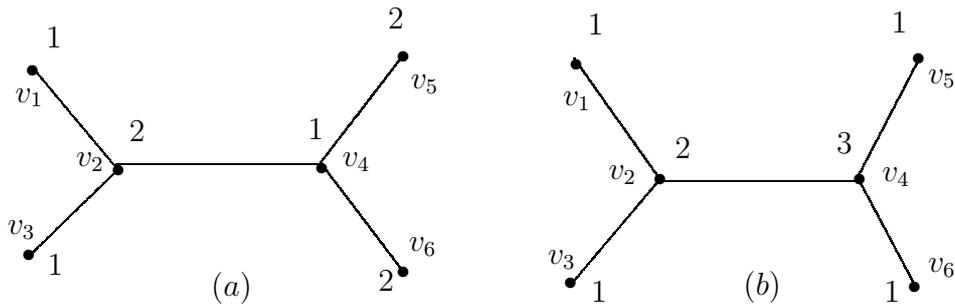
\unitlength 1mm 
\linethickness{0.4pt}
\ifx\plotpoint\undefined\newsavebox{\plotpoint}\fi 
\begin{picture}(126.5,38.25)(0,0)
\put(8.25,32.75){\circle*{1.5}}
\put(19.5,19.5){\circle*{1.5}}
\put(7.75,8.25){\circle*{1.5}}
\put(46.25,19.75){\circle*{1.5}}
\put(57,34.5){\circle*{1.5}}
\put(57,6){\circle*{1.5}}
\multiput(8.5,32.75)(.0336391437,-.0405198777){327}{\line(0,-1){.0405198777}}
\multiput(19.5,19.5)(-.034457478,-.0337243402){341}{\line(-1,0){.034457478}}
\put(19.25,20.25){\line(1,0){26.75}}
\multiput(46,20.25)(.0336826347,.0441616766){334}{\line(0,1){.0441616766}}
\multiput(46.25,20.25)(.0336990596,-.0446708464){319}{\line(0,-1){.0446708464}}
\put(6,28.5){$v_1$}
\put(14,20){$v_2$}
\put(5,11){$v_3$}
\put(49,20){$v_4$}
\put(58,30.25){$v_5$}
\put(57.75,8.75){$v_6$}
\put(10,35.75){$1$}
\put(10.25,5.5){$1$}
\put(21,23.25){$2$}
\put(44.5,23.5){$1$}
\put(55,38.25){$2$}
\put(53.75,3.5){$2$}
\put(31.75,2.75){$(a)$}
\put(-12,-5){$s_1: (v_1,1),(v_2,2),(v_3,1)(v_4,1),(v_5,2),(v_6, 2)$}
\put(30,-12){$\chi_{s_1}(T)= 2$}
\put(79.75,33.5){\circle*{1.5}}
\put(90.75,18.25){\circle*{1.5}}
\put(79,4.5){\circle*{1.5}}
\put(117,18.25){\circle*{1.5}}
\put(124.75,34.25){\circle*{1.5}}
\put(124.25,4){\circle*{1.5}}
\multiput(79.5,34.25)(.0336826347,-.0479041916){334}{\line(0,-1){.0479041916}}
\multiput(90.75,18.25)(-.0336676218,-.0401146132){349}{\line(0,-1){.0401146132}}
\put(90.75,18){\line(1,0){25.75}}
\multiput(116.5,18)(.0336734694,.0653061224){245}{\line(0,1){.0653061224}}
\multiput(116.75,18.5)(.033632287,-.063901345){223}{\line(0,-1){.063901345}}
\put(78,28){$v_1$}
\put(84,18.25){$v_2$}
\put(76,7){$v_3$}
\put(120,18){$v_4$}
\put(126,31){$v_5$}
\put(126,6){$v_6$}
\put(81.25,37.5){$1$}
\put(81.75,2){$1$}
\put(92.75,21){$2$}
\put(114,21.75){$3$}
\put(121.25,37.75){$1$}
\put(120.25,2){$1$}
\put(101.75,3){$(b)$}
\put(75,-5){$s_2: (v_1,1),(v_3,1),(v_5,1)(v_6,1),(v_2,2),(v_4,3)$}
\put(99,-12){$\chi_{s_2}(T)= 3$}
\put(67.5,-18){$Figure$ $5$}
\end{picture}
\\\\\\\\

In Figure 5, the sequence $s_1$ is a chromatic, but $s_2$ is not.\\\\

\unitlength .8mm 
\linethickness{0.4pt}
\ifx\plotpoint\undefined\newsavebox{\plotpoint}\fi 
\begin{picture}(119.25,65.375)(0,0)
\put(6.25,20.75){\circle*{1}}
\put(25.5,20.75){\circle*{1}}
\put(15.625,20.75){\vector(1,0){.07}}\put(6,20.75){\line(1,0){19.25}}
\put(5.5,15.25){$v_1$}
\put(26.25,15.5){$v_2$}
\put(6,25){$1$}
\put(24.75,25.25){$2$}
\put(15.25,7.25){$(a)$}
\put(15.25,0){$D_1$}
\put(44.5,20.5){\circle*{1.5}}
\put(65.25,20){\circle*{1.5}}
\put(55,37.25){\circle*{1.5}}
\put(54.75,20.25){\vector(1,0){.07}}\put(44.75,20.25){\line(1,0){20}}
\put(59.625,29){\vector(-2,3){.07}}\multiput(64.75,20.25)(-.0337171053,.0575657895){304}{\line(0,1){.0575657895}}
\put(49.5,29.125){\vector(2,3){.07}}\multiput(44.25,20.75)(.0336538462,.0536858974){312}{\line(0,1){.0536858974}}
\put(42.5,15.75){$v_1$}
\put(66,15.5){$v_2$}
\put(55,42.75){$v_3$}
\put(53.75,31){$3$}
\put(48,22.5){$1$}
\put(59,21.5){$2$}
\put(54.25,7.75){$(b)$}
\put(54.25,0){$D_2$}
\put(83.5,20.75){\circle*{1.5}}
\put(106.25,21.25){\circle*{1.5}}
\put(94.25,39.5){\circle*{1.5}}
\put(95,21){\vector(1,0){.07}}\put(83.75,21){\line(1,0){22.5}}
\put(100.125,30.625){\vector(-2,3){.07}}\multiput(106.25,21)(-.0336538462,.0528846154){364}{\line(0,1){.0528846154}}
\put(88.75,30.5){\vector(2,3){.07}}\multiput(83.25,21.25)(.0336391437,.0565749235){327}{\line(0,1){.0565749235}}
\put(115,43.5){\circle*{1}}
\put(88,48.75){\vector(4,3){.07}}\qbezier(83.25,21)(76.875,65.375)(115,43.25)
\put(115.125,43.625){\vector(1,3){.07}}\multiput(115,43.25)(.03125,.09375){8}{\line(0,1){.09375}}
\put(104.75,41.5){\vector(4,1){.07}}\multiput(94.5,39.75)(.197115385,.033653846){104}{\line(1,0){.197115385}}
\put(110.625,32.625){\vector(1,2){.07}}\multiput(106,21.75)(.0336363636,.0790909091){275}{\line(0,1){.0790909091}}
\put(83.75,16.25){$v_1$}
\put(107.25,16.5){$v_2$}
\put(93.75,43.25){$v_3$}
\put(119.25,44){$v_4$}
\put(88,23.5){$1$}
\put(99.5,23){$2$}
\put(93.75,33.5){$3$}
\put(118,39.75){$4$}
\put(94.5,7.75){$(c)$}
\put(94.5,0){$D_3$}
\put(50,-10){Figure $6$}
\end{picture}
\\\\
The following result is well known.\\
\begin{prop}Given a positive integer $n$, there exists a triangle free graph $G$ whose chromatic number is $n$.\end{prop}
 A similar result for digraphs is the following:\\
\begin{prop} Given a positive integer $n$, there exists a digraph $D$ and a sequence $s$ of its vertices in $D$ such that its $s$-dichromatic number $\chi_{sd}(D)= n$.\end{prop}
\begin{proof} Let $D_1$, be the digraph which consists of an arc $v_1v_2$.
We color $v_1$ with color 1, and then color $v_2$ with color 2 as in Figure 6(a). Take a new vertex $v_3$, and draw arcs $v_1v_3$ and $v_2v_3$ to get the digraph $D_2$. Now color the vertex $v_3$ with color 3 as shown in Figure 6(b). After having successively constructed the digraphs $D_1,D_2,D_3,\ldots,D_{n-1}$, and giving colors $1,2,3,\ldots,n-1$ to the vertices $v_1,v_2,\ldots,v_{n-1}$, we construct the digraph $D_n$ by taking a new vertex $v_n$, and adding arcs $v_iv_n$, for $1 \leq i \leq n-1$, and give color $n$ to $v_n$. The digraph $D_n$ thus constructed has $s$-coloring $s: (v_1,1),(v_2,2),\ldots,(v_n,n)$. Clearly, $\chi_{sd}(D) = n$.\end{proof}
\section{Complete Partition}
\qquad A partition $P=\{V_1,V_2,\ldots,V_k\}$ of the vertex set $V$ of a graph $G=(V,E)$ is \emph{complete} if there exists an edge between any two sets in $P$. A minimum order of a partition of $V$ into independent sets is called a \emph{chromatic partition}. It is well known that any chromatic partition is complete. A $s$-coloring $s:(v_1,c_1),(v_2,c_2),\ldots,(v_p,c_p)$ of a graph/digraph also partitions its vertex set $V=\{v_1,v_2,\ldots,v_p\}$ into independent/$c$-independent sets. In particular, a $s$-coloring of a graph $G=(V,E)$ with minimum number of colors gives a partition of the vertex set which is complete.\\

The \emph{achromatic number} $\psi(G)$ of a graph $G=(V,E)$ is the maximum order of a partition of $V$ into independent sets, which is complete. It now follows that for any graph $G$,
$$\chi(G)\leq\chi_s(G)\leq \psi(G) \hskip.5cm ...(2)$$

A \emph{complete coloring} of a graph $G$ is a partition of its vertex set $V$ into independent sets which is complete. It is known that for every integer \emph{a} with $\chi(G)\leq a \leq \psi(G)$, there exists a complete partition of $V$ of order $a$. Hence it follows that for any such integer $a$ there exists a $s$-coloring of $G$ such that $\chi_s(G)= a$. \\
 It is known that many upper bounds of $\chi(G)$ are also bounds for $\psi(G)$, and hence bounds for $\chi_s(G)$. In a $s$-coloring, if $\beta_{0s}$ is the maximum cardinality of a color class, one can easily make out that $p/\beta_{0_s}$ is a lower bound for $\chi_s(G)$.\\\\
\textbf{Problem}: Characterize graphs $G$ for which \\
$$\chi(G) = \chi_s(G) = \psi(G).$$
For example, complete graphs, wheels and cycles belong to this category.\\

A partition $P=\{V_1,V_2,\ldots,V_k\}$ of the vertex set $V$ of a digraph is \emph{complete} if there exists an arc between any two sets in $P$. A minimum order of a partition of $V$ into $c$-independent sets is a \emph{dichromatic partition} of $V$, and any such partition is complete.\\

The $s$-\emph{achrodichromatic number} $\psi_{sd}(D)$ of a digraph $D$ is the maximum order of a partition of $V$ into $c$-independent sets which is complete.\\

Let $V$ be the set of vertices of a graph/digraph, and $s(V)$ be the set of all sequences of vertices in $V$. \\

\textbf{Note:}\\
(i) For a graph $G=(V,E)$,\\
$$\chi(G)= min\{\chi_s(G): s \in s(V)\}$$

$$\psi(G) = max \{\chi_s(G):s \in s(V)\}.$$
(ii) For a digraph $D=(V,E)$,
$$\chi_d(D) = min\{\chi_{sd}(D) : s \in S(V)\}$$
$$\psi_{sd}(D)= max\{\chi_{sd}(D): s \in s(V)\}.$$

 For any digraph $D$, we have, $$\chi_d(D) \leq \chi_{sd}(D) \leq \psi_{sd}(D).$$
In Figure 7, $\psi_{sd}(D)=4$ where $s$-coloring is $s:(v_1,1) (v_2,2) (v_5,2) (v_3,3) (v_4,4).$ The above sequence also gives the $s$-achromatic number of the underlying graph $G$ of $D$.\\
\\\\
\unitlength .8mm 
\linethickness{0.4pt}
\ifx\plotpoint\undefined\newsavebox{\plotpoint}\fi 
\begin{picture}(98,85.5)(0,0)
\put(63,71.75){\circle*{1.5}}
\put(38,47.5){\circle*{1.5}}
\put(87.5,47.5){\circle*{1.5}}
\put(46.5,18){\circle*{1.5}}
\put(73.25,16.75){\circle*{1.5}}
\put(50.125,59.75){\vector(-1,-1){.07}}\multiput(62.25,72)(-.03372739917,-.03407510431){719}{\line(0,-1){.03407510431}}
\put(42.125,33){\vector(1,-4){.07}}\multiput(38,47.5)(.0336734694,-.1183673469){245}{\line(0,-1){.1183673469}}
\put(59.625,17.75){\vector(1,0){.07}}\multiput(46.25,18.5)(.59444444,-.03333333){45}{\line(1,0){.59444444}}
\put(75.375,59.75){\vector(1,-1){.07}}\multiput(63.25,71.75)(.03405898876,-.03370786517){712}{\line(1,0){.03405898876}}
\put(80.375,32.25){\vector(-1,-2){.07}}\multiput(87.5,47.75)(-.0336879433,-.073286052){423}{\line(0,-1){.073286052}}
\put(54.75,44.875){\vector(-1,-3){.07}}\multiput(63,71)(-.0336734694,-.1066326531){490}{\line(0,-1){.1066326531}}
\put(68,44.125){\vector(1,-4){.07}}\multiput(63,71.5)(.0336700337,-.1843434343){297}{\line(0,-1){.1843434343}}
\put(63.5,75){$v_1$}
\put(63.5,80){$1$}
\put(32,47){$v_2$}
\put(26,47){$2$}
\put(43.5,12){$v_3$}
\put(39.75,7){$3$}
\put(74.75,12){$v_4$}
\put(76.25,7){$4$}
\put(89,47){$v_5$}
\put(96,46){$2$}
\put(56,2){$Figure$ $7$}
\end{picture}
\\

\begin{prop} Let $G(D)$ be the underlying graph of digraph $D$. Then $$\psi_{sd}(D) \leq \psi_s(G(D)).$$ \end{prop}
\begin{proof} Let $P=\{V_1,V_2,\ldots,V_k\}$ be any $s$-achromatic partition of the vertex set $V$ of $G(D)$. Since each independent set is $c$-independent, $P$ is also a partition of $V$ into $c$-independent sets in $D$. Further, since $P$ is complete partition of $V$ in $G$, it is also a complete partition of $V$ in $D$. Hence, $\psi_{sd}(D) \leq \psi_s(G(D)).$\end{proof}

Let $D$ be a digraph and $G$ be its underlying graph of $D$. We have the following chain of relations concerning the coloring numbers discussed above. $$\chi_d(D) \leq \chi_s(D) \leq \chi(G) \leq \chi_s(G) \leq \psi_s(G) \leq \psi_{sd}(D).$$

\section{Coloring of a Digraph and its Matrix Representation}
\qquad A vertex labeled digraph $D$ can be uniquely represented by a matrix called the $L$-matrix of $D$. A characterization of such matrices is obtained. Further, a characterization of acyclic digraphs in terms of its $L$-matrix is obtained.\\\\

Let $D=(V,E)$ be a digraph and $V=\{v_1,v_2,\ldots,v_p\}$. Suppose the vertices of $D$ are labeled, and for a vertex $v_i$, let $\ell(v_i)$ be the label of $v_i$, $1 \leq i \leq p$. The labels of the vertices need not be distinct. One can represent such a labeled digraph $D$ uniquely by means of a square matrix $M=[a_{ij}]$ of order $p$ called the $L$-\emph{matrix} of $D$  whose entries $a_{ij}$ are defined as follows:\\

\begin{equation*}
a_{ij}=
\begin{cases}
 2, \text{if $v_iv_j$ is an arc and $\ell(v_i)=\ell(v_j)$}, \\
        1,  \text{if $v_iv_j$ is an arc and $\ell(v_i)\neq \ell(v_j)$},\\
       -1,  \text{if $v_iv_j$ is not an arc and $\ell(v_i)=\ell(v_j)$},\\
        0,  \text{otherwise}.
\end{cases}\end{equation*}
Note that the $L$- matrix is a non-symmetric with zero diagonal.\\\\
\unitlength .8mm 
\linethickness{0.08pt}
\ifx\plotpoint\undefined\newsavebox{\plotpoint}\fi 
\begin{picture}(109.5,78)(0,0)
\put(19.125,44){\vector(1,0){.07}}\put(2.5,44){\line(1,0){33.25}}
\put(19.75,70.25){\circle*{1.5}}
\put(2.5,44.25){\circle*{1.5}}
\put(35.25,44.25){\circle*{1.5}}
\put(60,77.25){\line(0,-1){40.75}}
\put(50.75,73.5){\line(1,0){58.75}}
\put(65.25,76.5){$v_1$}
\put(83.75,76.5){$v_2$}
\put(103.75,76.5){$v_3$}
\put(53.5,67){$v_1$}
\put(53.5,54.5){$v_2$}
\put(53.5,41.75){$v_3$}
\put(65.25,67){$0$}
\put(83.75,67){$1$}
\put(102,67){$-1$}
\put(65.25,54.5){$0$}
\put(83.75,54.5){$0$}
\put(103.75,54.5){$1$}
\put(65.25,41.75){$2$}
\put(83.75,41.75){$0$}
\put(103.75,41.75){$0$}
\put(19.5,75.5){$1$}
\put(1.25,38.25){$2$}
\put(35.5,38.25){$1$}
\put(19,63.25){$v_1$}
\put(7,47){$v_2$}
\put(-13,58){$D$:}
\put(47,27){$Figure$ $8$}
\put(11.125,57.25){\vector(-2,-3){.07}}\multiput(19.75,70)(-.0336914063,-.0498046875){512}{\line(0,-1){.0498046875}}
\put(29,47.5){$v_3$}
\put(27.875,57.75){\vector(-2,3){.07}}\multiput(35.75,45)(-.0337259101,.0546038544){467}{\line(0,1){.0546038544}}
\end{picture}\\
In Figure 8, the $L$-matrix of the labeled digraph $D$ is given.
Suppose there are $k$ distinct labels in the labeled digraph $D$. One can partition the vertex set $V$ into $k$ sets $V_1,V_2,\ldots,V_k$, such that all vertices in $V_i$ for $1\leq i\leq k$, have the same label.
Given the $L$-matrix $M= [a_{ij}]$ of a labeled digraph $D$, one can easily determine the digraph $D$ uniquely as well as the partition $V=\{V_1,V_2,\ldots,V_k\}$ of the vertex set given by the labels as follows:\\
Let $V=\{v_1,v_2,\ldots,v_p\}$, be the vertex set of $D$. For any two vertices $v_i$ and $v_j$ in $D$,
 $v_iv_j$ is an arc in $D$ if, and only if, $a_{ij}\in \{2,1\}$. Further, $v_i$ and $v_j$ have the same label if, and only if, $a_{ij}\in\{2,-1\}$. For a vertex $v_i$, the set of all vertices in $D$ having the same label as $v_i$ is determined as follows:\\
 $C(v_i)= \{v_i\} \cup\{v_j: a_{ij} \in \{2-1\}\}$.\\

In \cite{es} the authors have defined the $L$-matrix $M=[a_{ij}]$ of a vertex labeled graph $G=(V,E)$ of order $p$ analogously as follows:
Let $V=\{v_1,v_2,\ldots,v_p\}$. For any two vertices $v_i$ and $v_j$ in $G$, $v_iv_j$ is an edge if, and only if, $a_{ij}\in \{2,1\}$. Further, $\ell(v_i)=\ell(v_j)$ if, and only if, $a_{ij}=\{2,-1\}$. Note that this matrix is symmetric.\\
We now characterize $L$-matrix of a labeled digraph.
\begin{theorem}(Characterization) A square matrix $M=[a_{ij}]$ of order $p$ with zero diagonal and $a_{ij} \in \{2,1,0,-1\}$, is the $L$-matrix of a labeled digraph $D$ of order $p$, if, and only if, the following conditions hold.\\
i) $a_{ij}$,$a_{jk} \in \{2,-1\}$, implies $a_{ik}\in \{2,-1\}$.\\
ii) $a_{ij} \in \{2,-1\}$ and $a_{jk}\in \{0,1\}$, implies $a_{ik}\in \{0,1\}$.\end{theorem}
\begin{proof}{\emph{Necessity}:} Let $M=[a_{ij}]$ be the $L$-matrix of a labeled directed graph $D$ of order $p$ . If three vertices  $v_i$, $v_j$,and $v_k$ in $D$ are such that, $\ell(v_i) = \ell(v_j)$, and $\ell(v_j) = \ell(v_k)$ then $\ell(v_i) = \ell(v_k)$. In otherwords, if $a_{ij}$, $a_{jk} \in \{2, -1\}$ then $a_{ik}\in \{2,-1\}$. Thus (i) holds. Suppose now,
$\ell(v_i) = \ell(v_j)$, and $\ell(v_j) \neq \ell(v_k)$. Then $\ell(v_i) \neq \ell(v_k)$. In otherwords, $a_{ij}\in\{2,-1\}$ and $a_{jk}\in\{0,1\}$, then $a_{ik}\in \{0,1\}$ and (ii) holds.\\\\
{\emph{Sufficiency}}: Let $M=[a_{ij}]$ be a matrix of order $p$ satisfying the conditions (i) and (ii). One can easily construct a vertex labeled digraph $D$ with vertex set $V=\{v_1,v_2,\ldots,v_p\}$ such that $M$ is the $L$-matrix of $D$ as follows:\\

For any two vertices $v_i$, $v_j$ in $V$, $v_iv_j$ is an arc if, and only if, $a_{ij}\in\{2,1\}$. Further $\ell(v_i)=\ell(v_j)$ if, and only if, $a_{ij}\in\{2,-1\}$. One can easily see that $M$ is the $L$-matrix of the labeled digraph $D$ thus constructed.\end{proof}
\section{Digraph Coloring and $L$-Matrices}
\qquad The \emph{color matrix} of a colored digraph $D$ is nothing but the $L$-matrix of a vertex labeled digraph $D$.\\
\qquad \hskip 1cm Given a coloring of a digraph $D$, one can easily find its color matrix $M=[a_{ij}]$. Conversely, given the color matrix of $D$, one can find the digraph, and the partition of the vertex set into color classes. By Theorem 7, for an acyclic digraph $D$ , $\chi_d(D)= 1$.\\
The following is a characterization of an acyclic digraph $D$ in terms of its color matrix. We omit the proof.\\
\begin{prop} A square matrix $M=[a_{ij}]$ of order $p$ is the color matrix of an acyclic digraph $D=(V,E)$ of order $p$ if, and only if,\\
(i) for all $i$, $j$, $i\neq j $, $a_{ij}\in \{2,-1\}$.\\
(ii) $a_{ij} =2 \Leftrightarrow a_{ji}= -1$.\end{prop}

\end{document}